\documentclass[12pt]{article}
\usepackage{amsmath}
\usepackage{amssymb}
\usepackage{amsthm}
\usepackage{graphicx}
\usepackage{verbatim} 
\numberwithin{equation}{section}

\newtheorem{thm}{Theorem}[section]
\newtheorem{cor}[thm]{Corollary}
\newtheorem{lem}[thm]{Lemma}

\begin{document}

\title{Gaussian upper bounds for heat kernels of continuous time simple random walks}
\author{Matthew Folz\thanks{Department of Mathematics, The University of British Columbia, 1984 Mathematics Road, Vancouver, B.C., Canada, V6T 1Z2.  {\tt mfolz@math.ubc.ca}.  Research supported by an NSERC Alexander Graham Bell Canada Graduate Scholarship.}}
\date{}
\maketitle

\begin{abstract}
\noindent We consider continuous time simple random walks with arbitrary speed measure $\theta$ on infinite weighted graphs.  Write $p_t(x,y)$ for the heat kernel of this process.  Given on-diagonal upper bounds for the heat kernel at two points $x_1,x_2$, we obtain a Gaussian upper bound for $p_t(x_1,x_2)$.  The distance function which appears in this estimate is not in general the graph metric, but a new metric which is adapted to the random walk.  Long-range non-Gaussian bounds in this new metric are also established.  Applications to heat kernel bounds for various models of random walks in random environments are discussed.   
\end{abstract}

\section{Introduction}

Let $\Gamma = (G,E)$ be an unoriented graph.  We assume that $\Gamma$ is connected, contains neither loops nor multiple edges, is locally finite, and countably infinite.  Let $d$ be the usual graph metric; given $x,y\in G$, $d(x,y)$ is equal to the number of edges in the shortest (geodesic) path between $x$ and $y$.  We write $B(x,r):=\{y\in G:d(x,y)\leq r\}$ for the closed ball of radius $r$ in the metric $d$. \\

We assume that $\Gamma$ is a weighted graph, so that associated with each $(x,y)\in G\times G$ is a nonnegative edge weight $\pi_{xy}$ which is symmetric ($\pi_{xy}=\pi_{yx}$ for $x,y\in G$) and satisfies $\pi_{xy}>0$ if and only if $\{x,y\}\in E$.  The edge weights can be extended to a measure on $G$ by setting $\pi_x := \pi(\{x\}) := \sum_{y\in G} \pi_{xy}$ for $x\in G$, and this extends to all subsets of $G$ by countable additivity. \\

Let $(\theta_x)_{x\in G}$ be an arbitrary collection of positive vertex weights.  We consider the continuous-time simple random walk $(X_t)_{t\geq 0}$, which has generator $\mathcal{L}_\theta$, given by

\begin{align*}
(\mathcal{L}_\theta f)(x) := \frac{1}{\theta_x}\sum_{y\sim x} \pi_{xy}(f(y)-f(x)).
\end{align*}
\\
Regardless of the choice of $(\theta_x)_{x\in G}$, the jump probabilities of these processes are $P(x,y) = \pi_{xy}/\pi_x$; the various walks corresponding to different choices of $(\theta_x)_{x\in G}$ will be time-changes of each other. \\

Two specific choices of the vertex weights $(\theta_x)_{x\in G}$ arise frequently.  The first is the choice $\theta_x := \pi_x$, which yields a process called the constant-speed continuous time simple random walk (CSRW).  The CSRW may also be constructed by taking a discrete-time simple random walk on $(\Gamma,\pi)$, which we denote by $(X_n)_{n\in\mathbb{Z}_{+}}$, together with an independent rate $1$ Poisson process $(N_t)_{t\geq 0}$; the CSRW is the process $Y_t := X_{N_t}$. \\

The second choice, $\theta_x \equiv 1$, yields a stochastic process referred to as the variable-speed continuous time simple random walk (VSRW).  This walk has the same jump probabilities as the CSRW, but instead of waiting for an exponentially distributed time with mean $1$ at a vertex $x$ before jumping, the VSRW waits for an exponentially distributed time with mean $\pi^{-1}_x$.  As discussed in \cite{B-JDD}, the VSRW may explode in finite time.  \\

Associated with the process $(X_t)_{t\geq 0}$ is a semigroup $(P_t)_{t\geq 0}$ defined by $(P_tf)(x) := \mathbb{E}^xf(X_t)$, and which possesses a density $p_t(x,y)$ with respect to the measure $\theta$, defined by

\begin{equation*}
p_t(x,y) := \frac{1}{\theta_y}\mathbb{P}^x(X_t = y).
\end{equation*}
\\
This function is also called the {\it heat kernel} of the process $(X_t)_{t\geq 0}$. \\ 

We discuss here an alternative construction of the heat kernel which will be used in Section 3; this closely follows the discussion in \cite{We}.  Let $(G_n)_{n\in\mathbb{Z}_{+}}$ be an increasing sequence of finite connected subsets of $G$ whose limit is $G$.  Given $U\subset G$, we denote the first hitting time of $U$ by $T_U := \inf\{s\geq 0:X_s\in U\}$. \\

For each $n\in\mathbb{Z}_+$, we define the killed heat kernel $p^{(G_n)}_t(x,y)$ by

\begin{equation*}
p^{(G_n)}_t(x,y) := \frac{1}{\theta_y}\mathbb{P}^x(X_t=y,T_{G\setminus G_n}> t).
\end{equation*}
\\
This object satisfies the following conditions:

\begin{equation*}
\begin{cases}
\displaystyle\frac{\partial}{\partial t} p^{(G_n)}_t(x,y) = (\mathcal{L}_\theta)_{y} p^{(G_n)}_t(x,y) &\text{if }x,y\in G_n, \\
p^{(G_n)}_t(x,y) = 0 &\text{if }x\in G\setminus G_n\text{ or }y\in G\setminus G_n, \\
p^{(G_n)}_t(x,y) \geq 0 &\text{for all $x,y\in G$}.
\end{cases}
\end{equation*}
\\
Furthermore, we have that for all $x,y\in G$ and $t>0$ and $n\in\mathbb{Z}_{+}$,

\begin{align*}
p^{(G_n)}_t(x,y) &\leq p^{(G_{n+1})}_t(x,y), \\
\lim_{n\to\infty} p^{(G_n)}_t(x,y) &= p_t(x,y).
\end{align*}
\\
We will also need a distance function on $G\times G$ which is adapted to the vertex weights $(\theta_x)_{x\in G}$; this will be the metric which appears in our heat kernel estimates.  In general, Gaussian upper bounds for the heat kernel do not hold if one only considers the graph metric, see Remark 6.6 of \cite{B-JDD} for an example.  Let $d_\theta(\cdot,\cdot)$ be a metric which satisfies

\begin{equation}
\begin{cases}
\displaystyle \frac{1}{\theta_x}\sum_{y\sim x}\pi_{xy}d^2_\theta(x,y) \leq 1 &\text{for all $x\in G$}, \\
d_\theta(x,y) \leq 1 &\text{whenever $x,y\in G$ and $x\sim y$}. \label{dtheta}
\end{cases}
\end{equation}
\\
It is not difficult to verify that such metrics always exist.  We write $B_\theta(x,r) := \{y\in G:d_\theta(x,y)\leq r\}$ for the closed ball of radius $r$ in the metric $d_\theta$; it should be noted that $B_\theta(x,r)$ may contain infinitely many points for some choices of $x\in G$ and $r>0$, or, equivalently, points arbitrarily far from $x$ in the graph metric.  Note that for the CSRW, the graph metric always satisfies both of the above conditions. \\

The use of metrics different from the graph metric in heat kernel estimates was initiated by Davies in \cite{D2}, and this metric is similar to the metrics considered there.  These metrics are closely related to the intrinsic metric associated with a given Dirichlet form; some details on the latter may be found in \cite{HSC2}.  Recent work using similar metrics includes \cite{B-JDD}, \cite{F+}, \cite{G+}, and \cite{K++}. \\

We will need the following condition: \\

{\bf Definition:} A monotonically increasing function $g:(a,b)\to (0,\infty)$ is $(A,\gamma)-$regular on $(a,b)$ $(A\geq 1,\gamma>1,0\leq a<b\leq \infty)$ if for all $a<t_1<t_2<\gamma^{-1}b$, the inequality

\begin{equation*}
\frac{g(\gamma t_1)}{g(t_1)} \leq A\frac{g(\gamma t_2)}{g(t_2)}
\end{equation*}
\\
holds. If $a=0$ and $b=\infty$, then we say that $g$ is $(A,\gamma)-$regular. \\

For appropriate values of $A$ and $\gamma$, this set of functions includes polynomial functions such as $ct^{d/2}$, exponential functions such as $c\exp(Ct^{\alpha})$, and various piecewise combinations of $(A,\gamma)-$regular functions such as $c_1t^{d_1/2}{\bf 1}_{(0,T]}+c_2t^{d_2/2}{\bf 1}_{(T,\infty)}$, where $c_1$ and $c_2$ are chosen to ensure that the resulting function is continuous. \\

Our work will assume that one has already obtained on-diagonal upper bound for the heat kernel at two points $x_1,x_2\in G$; that is, there are functions $f_1,f_2$ which are $(A,\gamma)-$regular on $(a,b)$ such that, for all $t>0$ and $i\in\{1,2\}$,

\begin{equation} 
p_t(x_i,x_i) \leq \frac{1}{f_i(t)}. \label{uhkb1}
\end{equation}
\\
On-diagonal bounds such as \eqref{uhkb1} have been studied in considerable detail in both discrete and continuous settings, and follow from a variety of analytic inequalities, such as a Sobolev inequality \cite{V}, a Nash inequality \cite{Ca+}, a log-Sobolev inequality \cite{D3}, or a Faber-Krahn inequality \cite{G-FK}.  Generally, these methods yield a uniform upper bound, valid for all $x\in G$.  In the present setting of graphs, one may also use isoperimetic inequalities on general graphs, or volume growth estimates in the particular case of Cayley graphs of groups; details are in \cite{B}, \cite{V+}, and \cite{W}. \\

In the context of Riemannian manifolds, Grigor'yan has shown that any Riemannian manifold $M$ which satisfies an on diagonal upper bound at two points $x,y\in M$ admits a Gaussian upper bound for the heat kernel $q_t(x,y)$.  His result is as follows: \\

{\bf Theorem A.} {\it \cite{G-GUB}
Let $x_1,x_2$ be distinct points on a smooth Riemannian manifold $M$, and suppose that there exist $(A,\gamma)-$regular functions $f_1,f_2$ such that, for all $t>0$ and $i\in\{1,2\}$,

\begin{equation}
q_t(x_i,x_i) \leq \frac{1}{f_i(t)}. \label{ubxy1}
\end{equation}
\\
Then for any $D>2$ and all $t>0$, the Gaussian upper bound

\begin{equation}
q_t(x_1,x_2) \leq \frac{4A}{(f_1(\delta t)f_2(\delta t))^{1/2}}\exp\left(-\frac{d^2(x_1,x_2)}{2Dt}\right)
\end{equation}
\\
holds, where $\delta=\delta(D,\gamma)$.} \\

One remarkable aspect of this result is that it only requires on-diagonal bounds at the points $x_1$ and $x_2$.  Prior to \cite{G-GUB}, there are several proofs of Gaussian upper bounds for the heat kernel on manifolds, but these papers involve more restrictive hypotheses on the underlying manifold, in addition to requiring on-diagonal heat kernel estimates which hold for all $x\in G$.  In practice, the upper bounds \eqref{ubxy1} are often obtained from a uniform upper heat kernel bound using the techniques described previously, such as a Nash inequality.  However, Theorem A leaves open the possibility of obtaining Gaussian upper bounds for $q_t(x_1,x_2)$ using only the restricted information in \eqref{ubxy1}. \\

For the discrete time SRW on $(\Gamma,\pi)$, one may again assume a uniform upper bound for the heat kernel, and obtain a Gaussian upper bound from it.  This was done first by Hebisch and Saloff-Coste in \cite{HSC} using functional-analytic techniques, and later by Coulhon, Grigor'yan, and Zucca in \cite{C+}, using techniques analogous to the ones used by Grigor'yan in \cite{G-GUB}. \\

In discrete time, a SRW cannot move further than distance $n$ in time $n$, and hence $p_n(x,y)=0$ whenever $d(x,y)>n$, whereas a continuous time random walk has no such constraint.  For the CSRW on $\mathbb{Z}$ with the standard weights, the heat kernel does not exhibit Gaussian decay if $d(x,y)\gg t$ (see \cite{BB}), and as a result we will only attempt to obtain Gaussian upper bounds when $d_\theta(x,y)\leq t$.  Non-Gaussian estimates applicable where $d_\theta(x,y)\geq t$ will be discussed in Section 2, which adapt work of Davies from \cite{D1} and \cite{D2}. \\

Our main result is a Gaussian upper bound for the heat kernel $p_t(x,y)$ which is valid under mild hypotheses on $(\Gamma,\pi)$ and $(\theta_x)_{x\in G}$.  

\begin{thm} \label{HKUB1}
Let $(\Gamma,\pi)$ be a weighted graph, and suppose that there exists a constant $C_\theta>0$ such that the vertex weights $(\theta_x)_{x\in G}$ satisfy $\theta_x\geq C_\theta$ for each $x\in G$.  Let $f_1,f_2$ be $(A,\gamma)-$regular functions satisfying, for $i\in\{1,2\}$,

\begin{equation}\label{growth}
\sup_{0<t<\infty} \frac{f_i(t)}{e^{t^{1/2}}} \leq A.
\end{equation}
\\
Suppose also that there exist vertices $x_1,x_2\in G$ such that for all $t>0$ and $i\in\{1,2\}$,

\begin{equation}
p_t(x_i,x_i) \leq \frac{1}{f_i(t)} \label{uhkb3}.
\end{equation}
\\
Then there exist constants $C_1(A,\gamma,C_\theta),C_2(\gamma),\alpha(\gamma)>0$, such that for all $t\geq 1\vee d_\theta(x_1,x_2)$,

\begin{equation*}
p_t(x_1,x_2) \leq \frac{C_1}{(f_1(\alpha t)f_2(\alpha t))^{1/2}}\exp\left(-C_2\frac{d_\theta^2(x_1,x_2)}{t}\right).
\end{equation*}
\\
\end{thm}

{\bf Remarks:} \\

{\bf 1.} There is no assumption of stochastic completeness on the process $(X_t)_{t\geq 0}$; these heat kernel estimates hold even if $(X_t)_{t\geq 0}$ has finite explosion time. \\

{\bf 2.} The main utility of this result is in settings where $f_i(t)$ has polynomial growth, so that \eqref{growth} is satisfied.  Suppose that for $i\in \{1,2\}$, $f_i(t) = f(t) := \exp(ct^{\alpha})$ for some $c,\alpha>0$.  By Cauchy-Schwarz, $p_t(x_1,x_2) \leq (p_t(x_1,x_1)p_t(x_2,x_2))^{1/2}$, and hence $p_t(x_1,x_2) \leq \exp(-ct^\alpha)$ for all $t>0$.  On the other hand, by Theorem 2.2, if $C>1$ and $t\geq Cd_\theta(x_1,x_2)$,

\begin{equation*}
p_t(x_1,x_2) \leq (\theta_{x_1}\theta_{x_2})^{-1/2}\exp\left(-c_1\frac{d^2_\theta(x_1,x_2)}{t}\right).
\end{equation*}
\\
If $0\leq x\leq y$ and $0\leq x\leq z$, then $x\leq (yz)^{1/2}$, so for $t\geq Cd_\theta(x_1,x_2)$,

\begin{align*}
p_t(x_1,x_2) &\leq (\theta_{x_1}\theta_{x_2})^{-1/4}\exp\left(-\frac{c}{2}t^\alpha-c_1\frac{d^2_\theta(x_1,x_2)}{2t}\right) \\
&= \frac{c_2}{f(c_3 t)}\exp\left(-c_1\frac{d^2_\theta(x_1,x_2)}{2t}\right),
\end{align*}
\\
so that a Gaussian upper bound of the desired form can be obtained very easily.  Moreover, as $t\to\infty$, it is the on-diagonal term which provides most of the decay in the heat kernel and not the Gaussian exponential factor.\\

Nevertheless, the growth condition \eqref{growth} is satisfied in many applications (as long as $A$ is taken sufficiently large).  For example, it is typically satisfied for random walks on graphs of polynomial volume growth, super-polynomial but sub-exponential volume growth, or exponential volume growth satisfying a certain isoperimetric inequality \cite{W}. \\

{\bf 3.} Let us note that if $f$ is $(A_1,\gamma)-$regular, and $A_2\geq A_1\geq 1$, then $f$ is also $(A_2,\gamma)-$regular.  Thus, as long as there exist $A_1,A_2,A_3\geq 1$ such that $f_1$ is $(A_1,\gamma)-$regular, $f_2$ is $(A_2,\gamma)-$regular, and $\sup_{0<t<\infty} \frac{f_i(t)}{e^{t^{1/2}}} \leq A_3$, then for $A = A_1\vee A_2\vee A_3$, $f_1,f_2$ are $(A,\gamma)-$regular, and $\eqref{growth}$ is satisfied. \\

{\bf 4.} In many applications, one has a uniform on-diagonal heat kernel upper bound, that is, an estimate of the form

\begin{align*}
p_t(x,x) \leq \frac{1}{f(t)},
\end{align*}
\\
which is valid for all $x\in G$ and all $t>0$; various techniques for obtaining such estimates were discussed earlier.  However, in other cases, one may obtain a heat kernel upper bound of the form

\begin{align*}
p_t(x,x) \leq \frac{1}{V(x,ct^{1/2})},
\end{align*}
\\
which is valid for all $x\in G$ and all $t>0$, and where $c>0$ is independent of $x$ and $V(x,r):=\pi(B(x,r))$.  This particular on-diagonal upper bound is related to the condition of volume doubling; see \cite{Del}.  Theorem~\ref{HKUB1} yields Gaussian upper bounds for the heat kernel even in the second situation, where one may have a different on-diagonal upper bound at each point of the graph. \\

The following is an immediate consequence of Theorem~\ref{HKUB1}.

\begin{cor} \label{HKUB2}
Let $(\Gamma,\pi)$ be a weighted graph, and suppose that there exists a constant $C_\theta>0$ such that the vertex weights $(\theta_x)_{x\in G}$ satisfy $\theta_x\geq C_\theta$ for each $x\in G$.  Let $f$ be an $(A,\gamma)-$regular function satisfying \eqref{growth}.  If for each $t>0$, the uniform heat kernel condition

\begin{equation*}
\sup_{x\in G} p_t(x,x) \leq \frac{1}{f(t)}
\end{equation*}
\\
is satisfied, then there exist constants $C_1(A,\gamma,C_\theta),C_2(\gamma),\alpha(\gamma)>0$ such that for all $x_1,x_2\in G$, and $t\geq 1\vee d_\theta(x_1,x_2)$,

\begin{equation*}
p_t(x_1,x_2) \leq \frac{C_1}{f(\alpha t)}\exp\left(-C_2\frac{d_\theta^2(x_1,x_2)}{t}\right).
\end{equation*}
\\
\end{cor}

If $f$ is only $(A,\gamma)-$regular on $(T_1,T_2)$, then we obtain a restricted version of Theorem~\ref{HKUB1}:

\begin{thm} \label{HKUB3}
Let $(\Gamma,\pi)$ be a weighted graph, and suppose that there exists a constant $C_\theta>0$ such that the vertex weights $(\theta_x)_{x\in G}$ satisfy $\theta_x\geq C_\theta$ for each $x\in G$.  Let $f_1,f_2$ be $(A,\gamma)-$regular functions on $(T_1,T_2)$ satisfying, for $i\in\{1,2\}$,

\begin{equation*}
\sup_{t\in (T_1,T_2)} \frac{f_i(t)}{e^{t^{1/2}}} \leq A.
\end{equation*}
\\
If there exist vertices $v_1,v_2\in G$ such that for all $t\in (T_1,T_2)$ and $i\in\{1,2\}$, the estimate

\begin{equation*}
p_t(v_i,v_i) \leq \frac{1}{f_i(t)}
\end{equation*}
\\
holds, then there exist constants $C_1(A,\gamma,C_\theta),C_2(\gamma),\alpha(\gamma)>0$ such that for all $t>0$ satisfying $72\gamma^4e^4T_1^2 \vee 1 \vee d_\theta(v_1,v_2)<t<T_2$,

\begin{equation*}
p_t(v_1,v_2) \leq \frac{C_1}{(f_1(\alpha t)f_2(\alpha t))^{1/2}}\exp\left(-C_2\frac{d_\theta^2(v_1,v_2)}{t}\right).
\end{equation*}
\\
\end{thm}

{\bf Remarks:} \\

{\bf 1.}  The primary use of this result is in the case that $T_2=\infty$, in which case one obtains Gaussian upper bounds for all sufficiently large times.  In random environments such as supercritical percolation clusters, the functions which appear in existing on-diagonal heat kernel upper bounds may not be $(A,\gamma)-$regular, but rather $(A,\gamma)-$regular on $(T,\infty)$ for some $T>0$; Theorem~\ref{HKUB3} is useful for obtaining Gaussian upper bounds in this setting.   Theorem~\ref{HKUB3} has also been used to obtain Gaussian heat kernel estimates for the random conductance model; see \cite{A+}. \\

The structure of this paper is as follows.  Section 2 establishes long-range, non-Gaussian heat kernel upper bounds for the heat kernel using the metric $d_\theta$, similar to earlier estimates of Davies in \cite{D1} and \cite{D2}.  Sections 3 proves a maximum principle, analogous to the one established in \cite{G-GUB}; this is subsequently used to estimate a tail sum of the square of the heat kernel.  The direct analogue of the maximum principle from \cite{G-GUB} does not work in the setting of graphs, and additional restrictions are necessary in order to establish the maximum principle of this paper. \\

In Section 4, we estimate this tail sum further using a telescoping argument from \cite{G-GUB}.  In \cite{G-GUB}, this argument is iterated infinitely many times, but in the present setting the telescoping argument cannot be employed past a finite number of steps.  At this point, it is necessary to use the heat kernel estimates of Section 2 to get a final estimate on the tail sum.  In Section 5, this estimate of the tail sum is used to estimate a weighted sum of the square of the heat kernel, and in turn, this estimate is used in Section 6 to establish Theorem~\ref{HKUB1}.  Section 7 discusses the modifications to Section 4 which are necessary to prove Theorem~\ref{HKUB3}.  Finally, Section 8 discusses applications to random walks on percolation clusters, and how the results of this paper may be applied to existing work on random walks in random environments. \\


\section{Long range bounds for the heat kernel}

In this section, we establish non-Gaussian upper bounds for the heat kernel $p_t(x,y)$ which are close to optimal in the space-time region where $d_\theta(x,y)\geq t$.  These bounds are closely related to the long-range bounds found in \cite{D1} and \cite{D2}, and are established using the same general techniques.  These bounds hold for all $x,y\in G$ and all $t>0$, although they give results weaker than Gaussian upper bounds in the space-time region where $d_\theta(x,y)\leq t$. \\

\begin{thm}
If $x_1,x_2\in G$, then for all $t>0$,

\begin{equation*}
p_t(x_1,x_2) \leq (\theta_{x_1}\theta_{x_2})^{-1/2} \exp\left(-\frac{1}{2}d_\theta(x_1,x_2)\log\left(\frac{d_\theta(x_1,x_2)}{2et}\right)-\Lambda t\right),
\end{equation*}
\\
where $\Lambda\geq 0$ is the bottom of the $L^2$ spectrum of the operator $\mathcal{L}_\theta$. \\
\end{thm}

\begin{proof}
By Proposition 5 of \cite{D1}, for all $x,y\in G$ and $t>0$, we have the estimate

\begin{equation}
p_t(x,y) \leq (\theta_x\theta_y)^{-1/2}\inf_{\psi\in L^\infty(G)} \exp(\psi(x)-\psi(y)+c(\psi)t), \label{longrange}
\end{equation}
\\
where $c(\psi) := \sup_{x\in G} b(\psi,x)-\Lambda$, and

\begin{equation*}
b(\psi,x) := \frac{1}{2\theta_x}\sum_{y\sim x} \pi_{xy}(e^{\psi(y)-\psi(x)}+e^{\psi(x)-\psi(y)}-2).
\end{equation*}
\\
Fix $x_1,x_2\in G$, set $D:= d_\theta(x_1,x_2)$ and, for $\lambda>0$, define $\psi_\lambda(x) := \lambda (D\wedge d_\theta(x,x_1))\in L^\infty(G)$.  Using the triangle inequality for the metric $d_\theta$ and the fact that the function $g(t):= e^{t}+e^{-t} = 2\cosh(t)$ is increasing for $t\geq 0$, we obtain

\begin{align*}
b(\psi_\lambda,x) &:= \frac{1}{2\theta_x} \sum_{y\sim x} \pi_{xy}(e^{\psi(y)-\psi(x)}+e^{\psi(x)-\psi(y)}-2) \\
&\leq \frac{1}{2\theta_x} \sum_{y\sim x} \pi_{xy}(e^{\lambda d_\theta(x,y)}+e^{-\lambda d_\theta(x,y)}-2).
\end{align*}
\\
At this point, we use the inequality

\begin{equation*}
e^s+e^{-s}-2 \leq s^2e^{s},
\end{equation*}
\\
which is valid for all $s\geq 0$.  This gives

\begin{align*}
b(\psi_\lambda,x) &\leq \frac{1}{2\theta_x} \sum_{y\sim x} \pi_{xy}(e^{\lambda d_\theta(x,y)}+e^{-\lambda d_\theta(x,y)}-2) \\
&\leq \frac{1}{2\theta_x} \sum_{y\sim x} \pi_{xy}(\lambda^2d^2_\theta(x,y)e^{\lambda d_\theta(x,y)}) \\
&= \left(\frac{1}{\theta_x}\sum_{y\sim x} \pi_{xy}d^2_\theta(x,y)\right)\left(\frac{1}{2}\lambda^2e^{\lambda}\right) \\
&\leq \frac{1}{2}\lambda^2e^{\lambda}.
\end{align*}
\\
Since this estimate holds uniformly in $x$, we have that

\begin{equation*}
\sup_{x\in G} b(\psi_\lambda,x) \leq \frac{1}{2}\lambda^2e^{\lambda},
\end{equation*}
\\
and 

\begin{equation*}
c(\psi_\lambda) := \sup_{x\in G} b(\psi_\lambda,x) - \Lambda \leq \frac{1}{2}\lambda^2e^{\lambda} - \Lambda.
\end{equation*}
\\
Set $f(\lambda) := \frac{1}{2}\lambda^2e^{\lambda}$.  Combining these estimates with \eqref{longrange}, we get, for each $\lambda>0$,

\begin{align*}
p_t(x_1,x_2) &\leq (\theta_{x_1}\theta_{x_2})^{-1/2}\exp(\psi_\lambda(x_1)-\psi_\lambda(x_2)+c(\psi_\lambda)t) \\
&= (\theta_{x_1}\theta_{x_2})^{-1/2}\exp(-\lambda d_\theta(x_1,x_2)+c(\psi_\lambda)t) \\
&\leq (\theta_{x_1}\theta_{x_2})^{-1/2}\exp(-\lambda d_\theta(x_1,x_2)+f(\lambda)t - \Lambda t) \\
&= (\theta_{x_1}\theta_{x_2})^{-1/2}\exp\left(t\left(-\lambda\left(\frac{d_\theta(x_1,x_2)}{t}\right)+f(\lambda)\right)-\Lambda t\right).
\end{align*}
\\
By optimizing over $\lambda>0$, we have

\begin{equation*}
p_t(x_1,x_2) \leq (\theta_{x_1}\theta_{x_2})^{-1/2}\exp\left(t\widehat{f}\left(\frac{d_\theta(x_1,x_2)}{t}\right)-\Lambda t\right),
\end{equation*}
\\
where $\widehat{f}$ is the Legendre transform of $f$, defined by

\begin{equation*}
\widehat{f}(\gamma) := \inf_{\lambda>0} \left(-\lambda\gamma + f(\lambda)\right).
\end{equation*}
\\
Note that if $f(\lambda) \leq g(\lambda)$ for all $\lambda>0$, $\widehat{f}(\gamma) \leq \widehat{g}(\gamma)$.  Now, the function $g(\lambda) := e^{2\lambda}$ satisfies $f(\lambda) \leq g(\lambda)$ for all $\lambda>0$, so 

\begin{equation*}
\widehat{f}(\gamma) \leq \widehat{g}(\gamma) = -\frac{\gamma}{2}\log\left(\frac{\gamma}{2e}\right).
\end{equation*}
\\
Thus, applying this estimate to the preceding work gives

\begin{equation*}
p_t(x_1,x_2) \leq (\theta_{x_1}\theta_{x_2})^{-1/2} \exp\left(-\frac{1}{2}d_\theta(x_1,x_2)\log\left(\frac{d_\theta(x_1,x_2)}{2et}\right)-\Lambda t\right),
\end{equation*}
\\
which holds for all $t>0$. \\
\end{proof}

One may also use these results to obtain a weak Gaussian upper bound for the heat kernel which does not use any information from on-diagonal bounds.

\begin{thm}
If $x_1,x_2\in G$, then for $t\geq d_\theta(x_1,x_2)$,

\begin{equation*}
p_t(x_1,x_2) \leq (\theta_{x_1}\theta_{x_2})^{-1/2}\exp\left(-\frac{d^2_\theta(x_1,x_2)}{2t}\left(1-\frac{d_\theta(x_1,x_2)}{t}\right)-\Lambda t\right).
\end{equation*}
\\
\end{thm}

\begin{proof}
We proceed as in the proof of Theorem 2.1.  Instead of using the inequality $e^s+e^{-s}-2 \leq s^2e^s$, we use the estimate

\begin{equation*}
e^s+e^{-s}-2 \leq s^2\left(1+\frac{se^s}{6}\right),
\end{equation*}
\\
which was used previously in \cite{D2}; we then obtain estimates similar to those above, except with $f(\lambda) := \frac{1}{2}\lambda^2\left(1+\frac{\lambda e^\lambda}{6}\right)$.  In \cite{D2}, Davies computes that

\begin{equation*}
\widehat{(2f)}(\gamma) \leq -\frac{\gamma^2}{4}+\frac{\gamma^3}{8},
\end{equation*}
\\
and since $\widehat{f}(\gamma) = \frac{1}{2}\widehat{(2f)}(2\gamma)$, we obtain

\begin{equation*}
\widehat{f}(\gamma) \leq -\frac{1}{2}\gamma^2+\frac{1}{2}\gamma^3 = -\frac{1}{2}\gamma^2(1-\gamma).
\end{equation*}
\\
Inserting this estimate into the above yields

\begin{equation*}
p_t(x_1,x_2) \leq (\theta_{x_1}\theta_{x_2})^{-1/2}\exp\left(-\frac{d^2_\theta(x_1,x_2)}{2t}\left(1-\frac{d_\theta(x_1,x_2)}{t}\right)-\Lambda t\right),
\end{equation*}
\\
as desired. 
\end{proof}


\section{Maximum Principle} 

For the remainder of the paper, we fix a set of vertex weights $(\theta_x)_{x\in G}$ for which there exists $C_\theta>0$ with $\theta_x\geq C_\theta$ for all $x\in G$, and an associated metric $d_\theta$, satisfying \eqref{dtheta}.  We also fix an increasing set of finite connected subsets $(G_n)_{n\in\mathbb{Z}_{+}}$ with limit $G$. \\

Let $x_0\in G$ be a point for which there exists a $(A,\gamma)-$regular function $f$ satisfying \eqref{growth} such that for $t>0$,

\begin{equation*}
p_t(x_0,x_0) \leq \frac{1}{f(t)}.
\end{equation*}
\\
We define $u(x,t) := p_t(x_0,x)$, and $u^{(k)}(x,t) := p^{(G_k)}_t(x_0,x)$ for $k\in\mathbb{Z}_+$. \\

In this section, we will prove a maximum principle for the quantities

\begin{equation*}
J^{(k)}_R(t) := \sum_{x\in G_k} (u^{(k)})^2(x,t)\exp(\xi_R(x,t))\theta_x,
\end{equation*}
\\
where $\xi_R$ will be defined later.  This will allow us to estimate various sums and weighted sums of $u^2$.  One basic estimate which we will use repeatedly is, for any $H\subset G$ and $k\in\mathbb{Z}_+$,

\begin{equation}
\sum_{x\in H} (u^{(k)})^2(x,t)\theta_x \leq \sum_{x\in H} u^2(x,t)\theta_x \leq \sum_{x\in G}p_t(x_0,x)p_t(x,x_0)\theta_x= p_{2t}(x_0,x_0) \leq \frac{1}{f(2t)}, \label{usquared}
\end{equation}
\\
using the symmetry and semigroup properties of the heat kernel. \\

The reason for considering the killed heat kernels $p^{(G_k)}_t(x,y)$ is that the function $u^{(k)}$ is finitely supported, and thus there is no difficulty in interchanging double sums.  When

\begin{equation*}
\sup_{x\in G} \frac{\pi_x}{\theta_x} = \infty,
\end{equation*}
\\
$\mathcal{L}_\theta$ is not a bounded operator on $L^2(\theta)$ (see \cite{D1} for a proof), and the interchange of sums in \eqref{GG} is not straightforward.  We also remark that there is in general no simple description of the domain of the associated Dirichlet form $\mathcal{E}$ in this case. \\

Fix $k\in\mathbb{Z}_+$.  Differentiating $J^{(k)}_R(t)$ and using the fact that $u$ is a solution to the heat equation on $G_k$, we get (writing $u^{(k)}_x$ for $u^{(k)}(x,t)$, $\zeta$ for $\exp\circ\ \xi$, and $\zeta_x$ for $\zeta(x,t)$),

\begin{align*}
\frac{d}{dt}J^{(k)}_R(t) &= \sum_{x\in G} \left(\frac{\partial}{\partial t}u^{(k)}_x\right)(2u^{(k)}_x\zeta_x)\theta_x + \sum_{x\in G} \left(\frac{\partial}{\partial t}\zeta_x\right)(u^{(k)}_x)^2\theta_x \\
&= \sum_{x\in G_k} (\mathcal{L}_\theta u^{(k)}_x)(2u^{(k)}_x\zeta_x)\theta_x + \sum_{x\in G_k} \left(\frac{\partial}{\partial t}\zeta_x\right)(u^{(k)}_x)^2\theta_x.
\end{align*}
\\
Note that $\left(\frac{\partial}{\partial t}u^{(k)}_x\right)(2u^{(k)}_x\zeta_x) = (\mathcal{L}_\theta u^{(k)}_x)(2u^{(k)}_x\zeta_x)$ even if $x_0\not\in G_k$ or $x\not\in G_k$.  By a Gauss-Green type calculation and using the fact that $u^{(k)}_y = 0$ for $y\in G\setminus G_k$,

\begin{align}
\nonumber\sum_{x\in G_k} (\mathcal{L}_\theta u^{(k)}_x)(2u^{(k)}_x\zeta_x)\theta_x &= \sum_{x\in G_k}\sum_{y\in G}(u^{(k)}_y-u^{(k)}_x)(2u^{(k)}_x\zeta_x)\pi_{xy} \\
\nonumber&= \sum_{x\in G_k}\sum_{y\in G_k}(u^{(k)}_y-u^{(k)}_x)(2u^{(k)}_x\zeta_x)\pi_{xy} + \sum_{x\in G_k}\sum_{y\in G\setminus G_k}(u^{(k)}_y-u^{(k)}_x)(2u^{(k)}_x\zeta_x)\pi_{xy} \\
\nonumber&= \sum_{x\in G_k}\sum_{y\in G_k}(u^{(k)}_y-u^{(k)}_x)(2u^{(k)}_x\zeta_x)\pi_{xy} + \sum_{x\in G_k}\sum_{y\in G\setminus G_k}(-u^{(k)}_x)(2u^{(k)}_x\zeta_x)\pi_{xy} \\
\nonumber&\leq \sum_{x\in G_k}\sum_{y\in G_k}(u^{(k)}_y-u^{(k)}_x)(2u^{(k)}_x\zeta_x)\pi_{xy} \\
&= -\sum_{x,y\in G_k} (u^{(k)}_y-u^{(k)}_x)(u^{(k)}_y\zeta_y-u^{(k)}_x\zeta_x) \pi_{xy}. \label{GG}
\end{align}
\\
The equality \eqref{GG} follows from interchanging the order of summation, which is permissible since $u^{(k)}$ has finite support.  Completing the square, we see that

\begin{align*}
-\sum_{x,y\in G_k} (u^{(k)}_y-u^{(k)}_x)(u^{(k)}_y\zeta_y-u^{(k)}_x\zeta_x) \pi_{xy} = {}&-\sum_{x,y\in G_k} \zeta_y(u^{(k)}_y-u^{(k)}_x)^2\pi_{xy} \\
&- \sum_{x,y\in G_k} u^{(k)}_x(u^{(k)}_y-u^{(k)}_x)(\zeta_y-\zeta_x)\pi_{xy} \\
\leq {}& \frac{1}{4}\sum_{x,y\in G_k} (u^{(k)}_x)^2\frac{(\zeta_x-\zeta_y)^2}{\zeta_y}\pi_{xy}.
\end{align*}
\\
It follows that

\begin{align*}
\frac{d}{dt}J^{(k)}_R(t) &\leq \frac{1}{4}\sum_{x,y\in G_k} (u^{(k)}_x)^2\frac{(\zeta_x-\zeta_y)^2}{\zeta_y}\pi_{xy}+ \sum_{x\in G_k} \left(\frac{\partial}{\partial t}\zeta_x\right)(u^{(k)}_x)^2\theta_x \\
&= \sum_{x\in G_k} (u^{(k)}_x)^2\sum_{y\in G_k}\left(\frac{\theta_x}{\pi_x}\frac{\partial}{\partial t}\zeta_x+\frac{1}{4}\frac{(\zeta_x-\zeta_y)^2}{\zeta_y}\right)\pi_{xy} \\
&= \sum_{x\in G_k} (u^{(k)}_x)^2\zeta_x\sum_{y\in G_k}\left(\frac{\theta_x}{\pi_x}\frac{\partial}{\partial t}\xi_x+\frac{1}{4}\left(\frac{\zeta_x^2-2\zeta_x\zeta_y+\zeta_y^2}{\zeta_x\zeta_y}\right)\right)\pi_{xy} \\
&= \sum_{x\in G_k} (u^{(k)}_x)^2\zeta_x\sum_{y\in G_k}\left(\frac{\theta_x}{\pi_x}\frac{\partial}{\partial t}\xi_x+\frac{1}{2}(\cosh(\xi_x-\xi_y)-1)\right)\pi_{xy}.
\end{align*}
\\
Given $\lambda>1$, there exists $K_\lambda<\infty$ so that the inequality

\begin{equation}\label{cosh}
2\cosh t-2\leq \lambda t^2
\end{equation}
\\
holds for $|t|\leq K_\lambda$.  Now, we define the distance function $d_{R,\theta}(x) := (R-d_\theta(x_0,x))_+$, and set

\begin{equation*}
\xi_R(x,t) := -\frac{\delta d^2_{R,\theta}(x)+\varepsilon}{s-t}.
\end{equation*}
\\
Here $R\geq 0$, $t>0$, and $s=s(t)>t$ are parameters that will be allowed to vary, and $\delta,\varepsilon>0$ are parameters that will be fixed.  For the rest of this paper, we will fix $\lambda,\delta,\varepsilon$ so that the following conditions are satisfied:

\begin{align}
\lambda &> 1, \label{C1}\\
\delta&<\frac{1}{\lambda}, \label{C2}\\
\varepsilon &\geq \frac{\lambda\delta^2}{4(1-\lambda\delta)}, \label{C3}\\
\frac{K_\lambda}{\delta} &= 6\gamma e^2. \label{C4}
\end{align}
\\
Let us show that such an assignment of constants is possible by exhibiting $\lambda_0,\delta_0,\varepsilon_0$ which satisfy the above conditions.  First, we choose $\lambda_0=2$, so that $K_{\lambda_0}= 2.98\ldots\leq 3$; this satisfies \eqref{C1}.  Next, since $\lambda_0$ and $\gamma$ are known, we may define $\delta_0$ through \eqref{C4}, and estimate

\begin{equation*}
\delta_0 := \frac{K_{\lambda_0}}{6\gamma e^2} < \frac{1}{2\gamma e^2} < \frac{1}{\lambda_0},
\end{equation*}
\\
so that \eqref{C2} is also satisfied.  We then choose $\varepsilon_0$ to be 

\begin{equation*}
\varepsilon_0 := \frac{\lambda_0\delta_0^2}{4(1-\lambda_0\delta_0)}.
\end{equation*}
\\
Let us also note that \eqref{C3} is equivalent to

\begin{equation}
\frac{4\varepsilon}{\lambda\delta(\delta+4\varepsilon)} \geq 1. \label{C5}
\end{equation}
\\
Once $\lambda,\delta$ and $\varepsilon$ have been fixed, we have the following result:

\begin{lem}
{\bf (Maximum Principle)} If conditions \eqref{C1},\eqref{C2},\eqref{C3},\eqref{C4} are satisfied, and $R\geq 0$, $t>0$, and $s>t$ are chosen so that

\begin{equation} \label{rstineq}
R-6\gamma e^2(s-t)+\frac{1}{2}\leq 0,
\end{equation}
\\
then for each $k\in\mathbb{Z}_{+}$,
\begin{equation*}
\frac{\partial}{\partial t}J^{(k)}_R(t) \leq 0.
\end{equation*}
\\
\end{lem}
\begin{proof}
Given $k\in\mathbb{Z}_{+}$ and $x\in G_k$, set

\begin{equation*}
\phi^{(k)}(x) := \sum_{y\in G_k} \pi_{xy}\left(\frac{\theta_x}{\pi_x}\frac{\partial}{\partial t}\xi_x+\frac{1}{2}(\cosh(\xi_x-\xi_y)-1)\right)
\end{equation*}
\\
Suppose that for all $x\in G_k$, whenever $y\sim x$ and $y\in G_k$, $|\xi_x-\xi_y|\leq K_\lambda$.  Using \eqref{cosh}, the inequality $|d_{R,\theta}^2(x)-d_{R,\theta}^2(y)|\leq 2d_{R,\theta}(x)+1$, and \eqref{C1},\eqref{C2}, \eqref{C3}, and \eqref{C5}, we obtain

\begin{align*}
\phi^{(k)}(x) &:= \sum_{y\in G_k} \pi_{xy}\left(\frac{\theta_x}{\pi_x}\frac{\partial}{\partial t}\xi_x+\frac{1}{2}(\cosh(\xi_x-\xi_y)-1)\right) \\
&\leq \sum_{y\in G_k}\pi_{xy}\left(\frac{\theta_x}{\pi_x}\frac{d}{dt}\xi_x+\frac{\lambda}{4}(\xi_x-\xi_y)^2\right) \\
&= (s-t)^{-2}\sum_{y\in G_k}\pi_{xy}\left(-\frac{\theta_x}{\pi_x}(\delta d^2_{R,\theta}(x)+\varepsilon)+\frac{\lambda\delta^2}{4}(d^2_{R,\theta}(x)-d^2_{R,\theta}(y))^2\right) \\
&= (s-t)^{-2}\sum_{y\in G_k}\pi_{xy}\left(-\frac{\theta_x}{\pi_x}(\delta d^2_{R,\theta}(x)+\varepsilon)+\frac{\lambda\delta^2}{4}(d_{R,\theta}(x)-d_{R,\theta}(y))^2(d_{R,\theta}(x)+d_{R,\theta}(y))^2\right) \\
&\leq (s-t)^{-2}\sum_{y\in G_k}\pi_{xy}\left(-\frac{\theta_x}{\pi_x}(\delta d^2_{R,\theta}(x)+\varepsilon)+\frac{\lambda\delta^2}{4}d^2_\theta(x,y)(2d_{R,\theta}(x)+1)^2\right) \\
&= (s-t)^{-2}\left(-\theta_x(\delta d^2_{R,\theta}(x)+\varepsilon)+\sum_{y\in G_k}\pi_{xy}\frac{\lambda\delta^2}{4}d^2_\theta(x,y)(2d_{R,\theta}(x)+1)^2\right) \\
&= \frac{\lambda\delta^2}{4}(2d_{R,\theta}(x)+1)^2(s-t)^{-2}\theta_x\left(\frac{1}{\theta_x}\sum_{y\in G_k} d^2_\theta(x,y)\pi_{xy} - \frac{4}{\lambda\delta^2}\frac{\delta d^2_{R,\theta}(x)+\varepsilon}{(2d_{R,\theta}(x)+1)^2}\right) \\
&\leq \frac{\lambda\delta^2}{4}(2d_{R,\theta}(x)+1)^2(s-t)^{-2}\theta_x\left(\frac{1}{\theta_x}\sum_{y\in G_k} d^2_\theta(x,y)\pi_{xy} - \inf_{u\geq 0}\frac{4}{\lambda\delta^2}\frac{\delta u^2+\varepsilon}{(2u+1)^2}\right) \\
&= \frac{\lambda\delta^2}{4}(2d_{R,\theta}(x)+1)^2(s-t)^{-2}\theta_x\left(\frac{1}{\theta_x}\sum_{y\in G_k} d^2_\theta(x,y)\pi_{xy} - \frac{4\varepsilon}{\lambda\delta(\delta+4\varepsilon)}\right) \\
&\leq \frac{\lambda\delta^2}{4}(2d_{R,\theta}(x)+1)^2(s-t)^{-2}\theta_v\left(\frac{1}{\theta_x}\sum_{y\in G_k} d^2_\theta(x,y)\pi_{xy} - 1\right) \\
&\leq 0.
\end{align*}
\\
Since

\begin{equation*}
\frac{d}{dt} J^{(k)}_R(t) \leq \sum_{x\in G} (u^{(k)}_x)^2\zeta_x\phi^{(k)}(x),
\end{equation*}
\\
we conclude that

\begin{equation*}
\frac{d}{dt} J^{(k)}_R(t) \leq 0.
\end{equation*}
\\
Now, let us analyze the inequality 

\begin{equation*}
|\xi_x-\xi_y|=\left|\frac{\delta(d_{R,\theta}^2(x)-d_{R,\theta}^2(y))}{s-t}\right|\leq K_\lambda.
\end{equation*}
\\
As before, we have $|d_{R,\theta}^2(x)-d_{R,\theta}^2(y)|\leq 2d_{R,\theta}(x)+1$, so this holds if

\begin{equation*}
d_{R,\theta}(x) \leq \frac{K_\lambda}{2\delta}(s-t)-\frac{1}{2},
\end{equation*}
\\
and, since $d_{R,\theta}(x)\leq R$, it certainly holds when

\begin{equation*}
R-6\gamma e^2(s-t)+\frac{1}{2} \leq 0.
\end{equation*}
\\
which is precisely the condition in the statement of the Lemma. \\
\end{proof}

Now, for $k\in\mathbb{Z}_{+}$, we define

\begin{align*}
I^{(k)}_R(t) &:= \sum_{x\in G_k\setminus B_\theta(x_0,R)} (u^{(k)}(x,t))^2\theta_x, \\
I_R(t) &:= \sum_{x\in G\setminus B_\theta(x_0,R)} u^2(x,t)\theta_x.
\end{align*}
\\
By \eqref{usquared}, all of these quantities are finite, and by monotone convergence,

\begin{equation}
\lim_{k\to\infty} I^{(k)}_R(t) = I_R(t). \label{IMCT}
\end{equation}
\\
The maximum principle allows us to estimate $I$, as follows:

\begin{lem}
Suppose that $R_0\geq R_1$, and $s>t_0\geq t_1>0$ are such that $R,s,t$ satisfy \eqref{rstineq}.  Then

\begin{equation*}
I_{R_0}(t_0) \leq \exp\left(\frac{\varepsilon}{s-t_0}\right) I_{R_1}(t_1) + \exp\left(\frac{\varepsilon}{s-t_0}\right)\exp\left(-\frac{\delta(R_0-R_1)^2+\varepsilon}{s-t_1}\right)\frac{1}{f(2t_1)}.
\end{equation*}
\\
\end{lem}

\begin{proof}
First, since $d_{R_0,\theta}$ vanishes outside of $B_\theta(x_0,R_0)$, for each $k\in\mathbb{Z}_{+}$,

\begin{align*}
I^{(k)}_{R_0}(t_0) &:= \sum_{x\in G_k\setminus B_\theta(x_0,R_0)} (u^{(k)}(x,t_0))^2\theta_x \\
&\leq \sup_{x\in G_k\setminus B_\theta(x_0,R_0)} \exp(-\xi_{R_0}(x,t_0))\sum_{x\in G_k\setminus B_\theta(x_0,R_0)}(u^{(k)}(x,t_0))^2\exp(\xi_{R_0}(x,t_0))\theta_x \\
&\leq \exp\left(\frac{\varepsilon}{s-t_0}\right)\sum_{x\in G_k\setminus B_\theta(x_0,R_0)}(u^{(k)}(x,t_0))^2\exp(\xi_{R_0}(x,t_0))\theta_x \\
&\leq \exp\left(\frac{\varepsilon}{s-t_0}\right)J^{(k)}_{R_0}(t_0).
\end{align*}
\\
Next, for $\ell\in [t_1,t_0]$, 

\begin{equation*}
R_0-6\gamma e^2(s-\ell)+\frac{1}{2}\leq 0,
\end{equation*}
\\
and so the maximum principle yields $J^{(k)}_{R_0}(t_0)\leq J^{(k)}_{R_0}(t_1)$, so that

\begin{align*}
I^{(k)}_{R_0}(t_0) \leq {}& \exp\left(\frac{\varepsilon}{s-t_0}\right)J^{(k)}_{R_0}(t_1) \\
= {}& \exp\left(\frac{\varepsilon}{s-t_0}\right)\left(\sum_{x\in G_k\setminus B_\theta(x_0,R_1)}+\sum_{x\in G_k\cap B_\theta(x_0,R_1)}\right)(u^{(k)}(x,t_1))^2\exp(\xi_{R_0}(x,t_1))\theta_x \\
\leq {}& \exp\left(\frac{\varepsilon}{s-t_0}\right) I^{(k)}_{R_1}(t_1) \\
&+ \exp\left(\frac{\varepsilon}{s-t_0}\right)\sup_{x\in G_k\cap B_\theta(x_0,R_1)}\exp(\xi_{R_0}(x,t_1))\sum_{x\in G_k\cap B_\theta(x_0,R_1)} (u^{(k)}(x,t_1))^2\theta_x \\
= {}&\leq \exp\left(\frac{\varepsilon}{s-t_0}\right) I^{(k)}_{R_1}(t_1) + \exp\left(\frac{\varepsilon}{s-t_0}\right)\exp\left(-\frac{\delta(R_0-R_1)^2+\varepsilon}{s-t_1}\right)\frac{1}{f(2t_1)}.
\end{align*}
\\
The last three inequalities follow from bounding above the exponential weight $\exp(\xi_{R_0}(x,t_1))$ by $1$ (on $G_k\setminus B_\theta(x_0,R_1)$), by using the inequality $d_{R_0,\theta}(x) \geq R_0-R_1$ (on $G_k\cap B_\theta(x_0,R_1)$), and using \eqref{usquared}. \\

Letting $k\to\infty$ and using \eqref{IMCT}, we get 

\begin{equation*}
I_{R_0}(t_0) \leq \exp\left(\frac{\varepsilon}{s-t_0}\right) I_{R_1}(t_1) + \exp\left(\frac{\varepsilon}{s-t_0}\right)\exp\left(-\frac{\delta(R_0-R_1)^2+\varepsilon}{s-t_1}\right)\frac{1}{f(2t_1)},
\end{equation*}
\\
which completes the proof of the Lemma. \\
\end{proof}


\section{Further estimates for $I_R(t)$}

In this section, we will prove the following estimate for $I_R(t)$:

\begin{lem}
Suppose that $t_0\geq R_0\geq 1/2$.  There exist positive constants $m_0,m_1,n_0,n_1,\alpha$, which do not depend on either $t_0$ or $R_0$, so that

\begin{equation*}
I_{R_0}(t_0) \leq m_0\frac{1}{f(\alpha t_0)}\exp\left(-m_1\frac{R_0^2}{t_0}\right)+n_0\exp(-n_1R_0).
\end{equation*}
\\
\end{lem}

In \cite{G-GUB}, a similar estimate is obtained without the $n_0\exp(-n_1R_0)$ term, and is a key step in establishing Gaussian upper bounds.  The condition \eqref{growth} in the statement of Theorem 1.3 prevents the term $n_0\exp(-n_1R_0)$ from dominating the `Gaussian term' $m_0\frac{1}{f(\alpha t_0)}\exp\left(-m_1\frac{R_0^2}{t_0}\right)$. \\

\begin{proof}
Given $t_0\geq R_0\geq 1/2$, we define sequences $(t_j)_{j\in\mathbb{Z}_{+}}$, $(s_j)_{j\in\mathbb{Z}_{+}}$,$(R_j)_{j\in\mathbb{Z}_{+}}$ by 

\begin{align*}
t_j &:= t_0\gamma^{-j},\\
s_j &:= 2t_j, \\
R_j &:= \left(\frac{1}{2}+\frac{1}{j+2}\right)R_0.
\end{align*}
\\
Recall that $\gamma>1$ was seen first in the $(A,\gamma)-$ regularity of the function $f$.  Note that

\begin{align*}
R_j-R_{j+1} &\geq \frac{R_0}{(j+3)^2}, \\
s_j-t_{j+1} &= \left(2-\frac{1}{\gamma}\right)t_j.
\end{align*}
\\
As long as

\begin{equation}\label{rstcond}
R_j-6\gamma e^2(s_j-t_j)+\frac{1}{2}\leq 0,
\end{equation}
\\
then Lemma 3.2 gives

\begin{equation}\label{tele}
I_{R_j}(t_j)\leq \exp\left(\frac{\varepsilon}{s_j-t_j}\right)I_{R_{j+1}}(t_{j+1})+\frac{1}{f(2t_{j+1})}\exp\left(\frac{\varepsilon}{s_j-t_j}\right)\exp\left(-\frac{\delta(R_j-R_{j+1})^2+\varepsilon}{s_j-t_{j+1}}\right).
\end{equation}
\\
Let us analyze when \eqref{rstcond} is satisfied.  Let $j^*$ denote the maximal $j$ for which \eqref{rstcond} holds.  First, $j^*\geq 0$, since

\begin{equation*}
R_0-6\gamma e^2(s_0-t_0)+\frac{1}{2} = R_0 - 6\gamma e^2t_0 + \frac{1}{2} < 0
\end{equation*}
\\
Using the definition of $(R_j)_{j\in\mathbb{Z}_{+}}$, we obtain

\begin{align*}
\frac{1}{4}\leq \frac{R_0}{2}&< R_{j^*} \leq R_0, \\
\end{align*}
\\
and the maximality of $j^*$ shows that

\begin{align*}
R_{j^*}&\leq 6\gamma e^2t_{j^*}, \\
R_{j^*+1}&>6\gamma e^2t_{j^*+1}-\frac{1}{2}.
\end{align*}
\\
Rearranging, we obtain
\begin{align}
\nonumber \frac{1}{6\gamma e^2}R_{j^*}&\leq t_{j^*} < \frac{1}{2e^2}R_{j^*}, \\
\frac{1}{12\gamma e^2}R_{0}&< t_{j^*} < \frac{1}{2e^2}R_0. \label{starbounds}
\end{align}
\\
Applying \eqref{tele} repeatedly yields

\begin{align*}
I_{R_0}(t_0) \leq{} & \prod^{j^*}_{k=0}\exp\left(\frac{\varepsilon}{s_k-t_k}\right)I_{R_{j^*}}(t_{j^*})\\
&+\sum^{j^*}_{k=0}\left(\prod^k_{\ell=0}\exp\left(\frac{\varepsilon}{s_\ell-t_\ell}\right)\right)\exp\left(-\frac{\delta(R_k-R_{k+1})^2+\varepsilon}{s_k-t_{k+1}}\right)\frac{1}{f(2t_{k+1})} \\
:= {}&S_1+S_2.
\end{align*}
\\
The product in $S_1$ may be estimated as follows:

\begin{align}
\nonumber S_1 &:= \prod^{j^*}_{k=0}\exp\left(\frac{\varepsilon}{s_k-t_k}\right)I_{R_{j^*}}(t_{j^*}) \\
\nonumber &= \exp\left(\frac{\varepsilon}{t_0}\sum^{j^*}_{k=0}\gamma^{k}\right)I_{R_{j^*}}(t_{j^*}) \\
\nonumber &\leq \exp\left(\frac{\varepsilon\gamma}{\gamma-1}\frac{1}{t_{j^*}}\right)I_{R_{j^*}}(t_{j^*}) \\
\nonumber &\leq \exp\left(\frac{12\varepsilon\gamma^2e^2}{(\gamma-1)R_0}\right)I_{R_{j^*}}(t_{j^*}) \\
&\leq \exp\left(\frac{24\varepsilon\gamma^2e^2}{\gamma-1}\right)I_{R_{j^*}}(t_{j^*}). \label{e1}
\end{align}
\\
We will deal with the $I_{R_{j^*}}(t_{j^*})$ term later.  Continuing,

\begin{align*}
S_2 &:= \sum^{j^*}_{k=0}\left(\prod^{k}_{\ell=0}\exp\left(\frac{\varepsilon}{s_\ell-t_\ell}\right)\right)\exp\left(-\frac{\delta(R_k-R_{k+1})^2+\varepsilon}{s_k-t_{k+1}}\right)\frac{1}{f(2t_{k+1})} \\
&\leq \sum^{j^*}_{k=0}\exp\left(\frac{\varepsilon\gamma}{(\gamma-1)t_0}\gamma^k\right)\exp\left(-\frac{\delta(R_k-R_{k+1})^2+\varepsilon}{s_k-t_{k+1}}\right)\frac{1}{f(2t_{k+1})} \\
&= \sum^{j^*}_{k=0}\exp\left(\frac{\varepsilon\gamma^2}{(\gamma-1)(2\gamma-1)t_0}\gamma^k\right)\exp\left(-\frac{\delta(R_k-R_{k+1})^2}{s_k-t_{k+1}}\right)\frac{1}{f(2t_{k+1})} \\
&\leq \sum^{j^*}_{k=0}\exp\left(\frac{\varepsilon\gamma^2}{(\gamma-1)(2\gamma-1)t_0}\gamma^k\right)\exp\left(-\frac{\delta\gamma}{(2\gamma-1)}\frac{\gamma^k}{(k+3)^4}\frac{R_0^2}{t_0}\right)\frac{1}{f(2t_{k+1})}.
\end{align*}
\\
At this point, define $\beta>0$, which depends only on $\gamma>1$, by

\begin{equation*}
\beta := \inf_{k\geq 0}\frac{\gamma^{k+1}}{(2\gamma-1)(k+2)(k+3)^4},
\end{equation*}
\\
so that for $k\geq 0$,

\begin{equation*}
\beta(k+2) \geq \frac{\gamma^{k+1}}{(2\gamma-1)(k+3)^4}.
\end{equation*}
\\
The $(A,\gamma)-$regularity of $f$ gives, for $0\leq j\leq k$,

\begin{equation*}
\frac{f(2t_j)}{f(2t_{j+1})} \leq A\frac{f(2t_0)}{f(2t_1)},
\end{equation*}
\\
and multiplying these estimates together yields

\begin{align}\label{telemult}
\nonumber \frac{1}{f(2t_{k+1})} &\leq \frac{1}{f(2t_0)}\left(A\frac{f(2t_0)}{f(2t_1)}\right)^{k+1} \\
&= \frac{1}{f(2t_0)}\exp\left((k+1)\log\left(A\frac{f(2t_0)}{f(2t_1)}\right)\right).
\end{align}
\\
We remark that this is the only point in the proof where we use the $(A,\gamma)-$regularity of $f$. \\

Set $L := \log\left(A\frac{f(2t_0)}{f(2t_1)}\right)$ and insert \eqref{telemult} into our earlier estimate for $S_2$ to obtain 

\begin{align*}
S_2 \leq {}& \frac{1}{f(2t_0)}\sum^{j^*}_{k=0} \exp\left(\frac{\varepsilon\gamma^2}{(\gamma-1)(2\gamma-1)t_0}\gamma^k\right)\exp\left(-\delta \beta (k+2)\frac{R_0^2}{t_0}\right)\exp\left((k+1)L\right) \\
\leq {}& \frac{1}{f(2t_0)}\exp\left(\frac{\varepsilon\gamma^2}{(\gamma-1)(2\gamma-1)t_{j^*}}\right)\sum^{j^*}_{k=0}\exp\left(-\delta \beta (k+2)\frac{R_0^2}{t_0}\right)\exp\left((k+1)L\right) \\
= {}&\frac{1}{f(2t_0)}\exp\left(\frac{24\varepsilon\gamma^3e^2}{(\gamma-1)(2\gamma-1)}\right)\exp\left(-\delta \beta\frac{R_0^2}{t_0}\right) \\
&\times\sum^{j^*}_{k=0}\exp\left(-(k+1)\left(\delta \beta\frac{R_0^2}{t_0}-L\right)\right).
\end{align*}
\\
At this point, we divide into cases based on whether

\begin{equation*}
\delta \beta \frac{R_0^2}{t_0} - L \geq \log 2
\end{equation*}
\\
or not.  If it is, then we have 

\begin{align}
\nonumber S_2 &\leq \frac{1}{f(2t_0)}\exp\left(\frac{24\varepsilon\gamma^3e^2}{(\gamma-1)(2\gamma-1)}\right)\exp\left(-\delta \beta\frac{R_0^2}{t_0}\right)\sum^{j^*}_{k=0}\exp\left(-(k+1)\log 2\right) \\
&\leq \frac{1}{f(2t_0)}\exp\left(\frac{24\varepsilon\gamma^3e^2}{(\gamma-1)(2\gamma-1)}\right)\exp\left(-\delta \beta\frac{R_0^2}{t_0}\right).\label{e2}
\end{align}
\\
If not, then we can estimate $S_2$ by

\begin{align}
\nonumber S_2 &\leq I_{R_0}(t_0) \\
\nonumber &\leq \sum_{x\in G} u^2(x,t_0)\theta_x \\
\nonumber &\leq \frac{1}{f(2t_0)} \\
\nonumber &\leq \frac{1}{f(2t_0)}\exp\left(-\delta \beta \frac{R_0^2}{t_0} + \log\left(A\frac{f(2t_0)}{f(2t_1)}\right)+\log 2\right) \\
&= \frac{2A}{f(2t_1)}\exp\left(-\delta \beta \frac{R_0^2}{t_0}\right). \label{e3}
\end{align}
\\
It remains to estimate the quantity $I_{R_{j^*}}(t_{j^*})$.  From Theorem 2.1, we have the following pointwise estimate of the heat kernel:

\begin{align*}
p_t(x,y) \leq (\theta_x\theta_y)^{-1/2}\exp\left(-\frac{1}{2}d_\theta(x,y)\log\left(\frac{d_\theta(x,y)}{2et}\right)\right)
\end{align*}

Hence,

\begin{align*}
I_{R_{j^*}}(t_{j^*}) &:= \sum_{x\in G\setminus B_\theta(v_0,R_{j^*})} u^2(x,t_{j^*})\theta_x \\
&\leq \sup_{x\in G\setminus B_\theta(x_0,R_{j^*})} u(x,t_{j^*}) \sum_{x\in G\setminus B_\theta(x_0,R_{j^*})} u(x,t_{j^*})\theta_x \\
&\leq \sup_{x\in G\setminus B_\theta(x_0,R_{j^*})} u(x,t_{j^*}). \\
\end{align*}

At this point, note that if $t>0$ is fixed, the function

\begin{align*}
\phi_t(d) := \exp\left(-\frac{1}{2}d\log\left(\frac{d}{2et}\right)\right)
\end{align*}
\\
is nonincreasing for $d\geq 2t$.  Since $R_{j^*}> 2e^2t_{j^*}$, we get 

\begin{align}
\nonumber I_{R_{j^*}}(t_{j^*}) &\leq \sup_{x\in G\setminus B_\theta(x_0,R_{j^*})} u(x,t_{j^*}) \\
\nonumber &\leq C_\theta^{-1}\phi_{t_{j^*}}(R_{j^*}) \\
\nonumber &\leq C_\theta^{-1}\phi_{t_{j^*}}\left(2e^2t_{j^*}\right) \\
\nonumber &= C_\theta^{-1}\exp\left(-e^2t_{j^*}\right) \\
&\leq C_\theta^{-1}\exp\left(-\frac{1}{12\gamma}R_0\right). \label{e4}
\end{align}
\\
This is the only point in the argument at which we explicitly use the fact that the vertex weights are bounded below. \\

Now, we can put all of our estimates together.  Combining \eqref{e1},\eqref{e2},\eqref{e3},\eqref{e4} we have

\begin{equation*}
I_{R_0}(t_0) \leq m_0\frac{1}{f(\alpha t_0)}\exp\left(-m_1\frac{R_0^2}{t_0}\right)+n_0\exp(-n_1R_0),
\end{equation*}
\\
where the constants $\alpha,m_0,m_1,n_0,n_1$ may be taken to be

\begin{equation*}
\alpha := \frac{2}{\gamma}, \qquad m_0 := \exp\left(\frac{24\varepsilon\gamma^3e^2}{(\gamma-1)(2\gamma-1)}\right)\vee 2A, \qquad m_1 := \delta \beta,
\end{equation*}
\begin{equation*}
n_0 := C_\theta^{-1}\exp\left(\frac{24\varepsilon\gamma^2e^2}{\gamma-1}\right), \qquad n_1 := \frac{1}{12\gamma}.
\end{equation*}
\\
The fact that $\gamma-1$ can be very close to $0$ is a potential concern.  In practice, one will often have the choice of several values of $\gamma$; for example, if $f(t)=t^\alpha$, one may choose any $\gamma>1$.  One also has the option of using the fact that $(A,\gamma)-$regularity implies $(A^{2^n},\gamma^{2^n})-$regularity to increase $\gamma$ at the cost of increasing $A$ (and hence $m_0$) also.  However, choosing $\gamma$ excessively large will cause $\alpha$ and $n_1$ to be undesirably close to zero.

\end{proof}


\section{Estimating a weighted sum of $u^2$}

For $H\subset G$, let us define the following weighted sum of $u^2$,

\begin{align*}
E_{\kappa,D,H}(x_0,t) &:= \sum_{x\in H} u^2(x,t)\exp\left(\kappa\frac{(d_\theta(x,x_0)\wedge D)^2}{t}\right)\theta_x \\
&= \sum_{x\in H} p^2_t(x,x_0)\exp\left(\kappa\frac{(d_\theta(x,x_0)\wedge D)^2}{t}\right)\theta_x.
\end{align*}
\\
\begin{lem} There exist constants $\kappa_0,C,\alpha_0>0$ such that for $t\geq \frac{1}{2} \vee \frac{D}{2},$

\begin{equation*}
E_{\kappa_0,D,G}(x_0,t) \leq \frac{C}{f(\alpha_0 t)}.
\end{equation*}
\\
\end{lem}
\begin{proof}
Fix $t\geq \frac{1}{2} \vee \frac{D}{2}$, and choose $\kappa_0$ to satisfy the inequalities $16\kappa_0-m_1<0$, $8\kappa_0-n_1<0$, where $m_1,n_1$ are the constants in Lemma 4.1. \\

We define $k^*$ to be the largest nonnegative integer so that $2^{k^*}\leq \sqrt{t}$ (if there is no such nonnegative integer, set $k^*=0$), and partition $G$ as $\displaystyle \bigcup_{0\leq j\leq k^*+1} A_k$, where

\begin{align*}
A_{0} &:= \{x\in G:d_\theta(x_0,x)\leq \sqrt{t}\},\\
A_k &:= \{x\in G:2^{k-1}\sqrt{t} < d_\theta(x_0,x) \leq 2^{k}\sqrt{t}\} \ \text{for $1\leq k\leq k^*$}, \\
A_{k^*+1} &:= \{x\in G:d_\theta(x_0,x)> 2^{k^*}\sqrt{t}\}.
\end{align*}
\\
We turn our attention to the quantities $E_{\kappa_0,D,A_j}(x_0,t)$ for $0\leq j\leq k^*+1$, which satisfy

\begin{equation} \label{E1}
E_{\kappa_0,D,G}(x_0,t) = \sum^{k^*+1}_{j=0} E_{\kappa_0,D,A_j}(x_0,t).
\end{equation}
\\
On $A_0$, the exponential weight $\exp\left(\kappa_0\frac{(d_\theta(x,x_0)\wedge D)^2}{t}\right)$ is bounded above by $e^{\kappa_0}$, and hence 

\begin{equation}
E_{\kappa_0,D,A_0}(x_0,t) \leq e^{\kappa_0}\sum_{x\in A_0} u^2(x,t)\theta_x\leq e^{\kappa_0}\frac{1}{f(2t)}\leq e^{\kappa_0}\frac{1}{f(\alpha t)}. \label{E2}
\end{equation}

For $1\leq j\leq k^*$, on $A_j$, the exponential weight $\exp\left(\kappa_0\frac{(d_\theta(x,x_0)\wedge D)^2}{t}\right)$ is bounded above by $\exp(\kappa_04^j)$.  Since $2^{j-1}\sqrt{t} \leq t$, we may apply the bound of Lemma 4.1 to obtain \\

\begin{align*}
\sum^{k^*}_{j=1} E_{\kappa_0,D,A_j}(x_0,t) \leq {}&\sum^{k^*}_{j=1} \exp(\kappa_04^j)I_{2^j\sqrt{t}}(t) \\
\leq {}& \sum^{k^*}_{j=1} \exp(\kappa_04^j)\left(m_0\frac{1}{f(\alpha t)}\exp(-m_14^j)+n_0\exp(-n_12^{j-1}\sqrt{t})\right) \\
= {}& m_0\frac{1}{f(\alpha t)}\sum^{k^*}_{j=1}\exp((4\kappa_0-m_1)4^{j-1})+n_0\sum^{k^*}_{j=1} \exp(2^{j-1}(2\kappa_02^j-n_1\sqrt{t})) \\
\leq {}& m_0\frac{1}{f(\alpha t)}\sum^{k^*}_{j=1}\exp((4\kappa_0-m_1)4^{j-1})+n_0\sum^{k^*}_{j=1} \exp(2^{j-1}(4\kappa_0-n_1)\sqrt{t})) \\
\leq {}& m_0\frac{1}{f(\alpha t)}\sum^{k^*}_{j=1}\exp((4\kappa_0-m_1)4^{j-1}) \\
&+n_0\exp((4\kappa_0-n_1)\sqrt{t})\sum^{k^*}_{j=1} \exp((2^{j-1}-1)(4\kappa_0-n_1)\sqrt{t}) \\
\leq {}& m_0\frac{1}{f(\alpha t)}\sum^{k^*}_{j=1}\exp((4\kappa_0-m_1)4^{j-1}) \\
&+n_0\exp((4\kappa_0-n_1)\sqrt{t})\sum^{k^*}_{j=1} \exp\left(\frac{1}{\sqrt{2}}(2^{j-1}-1)(4\kappa_0-n_1)\right) \\
\leq {}& m_0T_0\frac{1}{f(\alpha t)} + n_0T_1\exp((4\kappa_0-n_1)\sqrt{t}),
\end{align*}
\\
where

\begin{align*}
T_0 &:= \sum^\infty_{j=1} \exp((4\kappa_0-m_1)4^{j-1})<\infty, \\
T_1 &:= \sum^{k^*}_{j=1} \exp\left(\frac{1}{\sqrt{2}}(2^{j-1}-1)(4\kappa_0-n_1)\right)<\infty.
\end{align*}
\\
By \eqref{growth}, we know that 

\begin{equation*}
\exp((4\kappa_0-n_1)\sqrt{t}) \leq \frac{A}{f((4\kappa_0-n_1)^2t)},
\end{equation*}
\\
so that

\begin{equation}\label{E3}
\sum^{k^*}_{j=1} E_{\kappa_0,D,A_j}(x_0,t) \leq (m_0T_0+n_0T_1A)\frac{1}{f((\alpha \vee (4\kappa_0-n_1)^2)t)}.
\end{equation}

On $A_{k^*+1}$, the exponential weight $\exp\left(\kappa_0\frac{(d_\theta(x,x_0)\wedge D)^2}{t}\right)$ is bounded above by $\exp\left(\kappa_0\frac{D^2}{t}\right)\leq \exp(4\kappa_0t)$, since $D\leq 2t$.  By definition, we have $\frac{1}{2}\sqrt{t} <2^{k^*}\leq \sqrt{t}$, and hence another application of Lemma 3.1 gives

\begin{align*}
E_{\kappa_0,D,A_{k^*+1}}(x_0,t) &\leq \exp(4\kappa_0t)I_{2^{k^*}\sqrt{t}}(t) \\
&\leq \exp(4\kappa_0t)I_{t/2}(t) \\
&\leq m_0\frac{1}{f(\alpha t)}\exp\left(4\kappa_0t-m_1\frac{t}{4}\right)+n_0\exp\left(4\kappa_0t-n_1\frac{t}{2}\right) \\
&= m_0\frac{1}{f(\alpha t)}\exp\left(\frac{1}{4}(16\kappa_0-m_1)t\right)+n_0\exp\left(\frac{1}{2}(8\kappa_0-n_1)t\right) \\
&= m_0\frac{1}{f(\alpha t)}\exp\left(\frac{1}{8}(16\kappa_0-m_1)\right)+n_0\exp\left(\frac{1}{2}(8\kappa_0-n_1)t\right) \\
&\leq m_0\frac{1}{f(\alpha t)}\exp\left(\frac{1}{8}(16\kappa_0-m_1)\right)+n_0\exp\left(\frac{1}{2\sqrt{2}}(8\kappa_0-n_1)\sqrt{t}\right). \\
\end{align*}
\\
By \eqref{growth} again,

\begin{equation*}
\exp\left(\frac{1}{2\sqrt{2}}(8\kappa_0-n_1)\sqrt{t}\right) \leq \frac{A}{f(1/8\cdot(8\kappa_0-n_1)^2t)},
\end{equation*}
\\
and so

\begin{equation} \label{E4}
E_{\kappa_0,D,A_{k^*+1}}(x_0,t) \leq (m_0\exp\left(\frac{1}{8}(16\kappa_0-m_1)\right)+n_0A)\frac{1}{f((\alpha \vee 1/8\cdot (8\kappa_0-n_1)^2t)}.
\end{equation}
\\
Combining \eqref{E1} with \eqref{E2},\eqref{E3}, and \eqref{E4} completes the proof. \\
\end{proof}

\section{Gaussian upper bounds for the heat kernel}

We are now ready to prove Theorem~\ref{HKUB1}. 

\begin{proof}

Let $D:= d_\theta(x_1,x_2)$ and assume that $t\geq 1\vee D$.  Then $\frac{t}{2}\geq \frac{1}{2}\vee \frac{D}{2}$, so we may apply Lemma 5.1 with the points $x_1$ and $x_2$ (for which we have \eqref{uhkb3}) to obtain positive constants $c$ and $\alpha$ such that, for $t\geq 1\vee D$,

\begin{align*}
E_{c,D,G}(x_1,t/2) &\leq \frac{C}{f_1(\alpha t/2)}, \\
E_{c,D,G}(x_2,t/2) &\leq \frac{C}{f_2(\alpha t/2)}.
\end{align*}
\\
The truncated distance $\rho_\theta(x,y):= d_\theta(x,y)\wedge D$ satisfies $d_\theta^2(x_1,x_2) = \rho_\theta^2(x_1,x_2) \leq 2(\rho_\theta^2(x_1,x)+\rho_\theta^2(x,x_2))$ for all $x\in G$.  By using the semigroup property and Cauchy-Schwarz combined with the above considerations, we obtain, for all $t\geq 1\vee D$,

\begin{align*}
p_t(x_1,x_2) &= \sum_{x\in G} p_{t/2}(x_1,x)p_{t/2}(x,x_2)\theta_x \\
&\leq \sum_{x\in G} p_{t/2}(x_1,x)\exp\left(c\frac{\rho_\theta^2(x_1,x)}{t}\right)p_{t/2}(x,x_2)\exp\left(c\frac{\rho_\theta^2(x_2,x)}{t}\right)\exp\left(-c\frac{\rho_\theta^2(x_1,x_2)}{2t}\right)\theta_x \\
&\leq (E_{c,D,G}(x_1,t/2)E_{c,D,G}(x_2,t/2))^{1/2} \exp\left(-c\frac{\rho_\theta^2(x_1,x_2)}{2t}\right) \\
&\leq \frac{C}{(f_1(\alpha t/2)f_2(\alpha t/2))^{1/2}}\exp\left(-c\frac{d^2_\theta(x_1,x_2)}{2t}\right),
\end{align*}
\\
which completes the proof of Gaussian upper bounds for the heat kernel. \\
\end{proof}


\section{Restricted $(A,\gamma)-$regular functions}

In Section 4, where we estimated the quantity $I_R(t)$, we assumed that $t_0\geq R_0\geq 1/2$, and used $(A,\gamma)-$regularity to obtain, for $0\leq k\leq j^*$,

\begin{equation*}
\frac{1}{f(2t_{k+1})} \leq \frac{1}{f(2t_0)}\left(A\frac{f(2t_0)}{f(2t_1)}\right)^{k+1}.
\end{equation*}
\\
This is the only point at which $(A,\gamma)-$regularity is used.  It follows that if $f$ is merely $(A,\gamma)-$regular on $(T_1,T_2)$, then for this inequality to hold, we must have $T_1<2t_{j^*+1}$ and $2t_1<\gamma^{-1}T_2$.  Subsequently, in Section 5, we apply our bounds for $I_R(t)$ with $t=t_0$ and $R=2^j\sqrt{t}$, for $0\leq j\leq \sup\{k\in\mathbb{Z}:2^k\leq \sqrt{t}\}\vee 0$.  Using \eqref{starbounds}, and setting $t_0=t/2$ (where $t\geq 1\vee D$), we see that these inequalities hold when

\begin{align*}
T_1&<\frac{1}{6\gamma^2e^2}(t/2)^{1/2}, \\
T_2&>2(t/2).
\end{align*}
\\
Rearranging, we have

\begin{align*}
t&> 72e^{4}\gamma^4T_1^2, \\
t&<T_2,
\end{align*}
\\
and applying these additional constraints yields Theorem~\ref{HKUB3}.

\section{Applications to random walks on percolation clusters}

In this section, we show how Theorem~\ref{HKUB3} may be used to obtain Gaussian upper bounds for the CSRW on the infinite component of supercritical bond percolation on the lattice $\mathbb{Z}^d$ equipped with the standard weights.  A detailed description of percolation is given in \cite{Gr}; a percolation cluster is a random connected subgraph of the lattice $\mathbb{Z}^d$ obtained by deleting each edge independently with probability $1-p$ and keeping it otherwise.  By fundamental results of percolation theory, there exists a critical probability $p_c(d)$ such that for $p>p_c(d)$ (i.e., the supercritical case), there is an a.s. unique infinite cluster; we consider the CSRW on this family of random graphs, which we denote by $\mathcal{C}_{p,\infty}(\omega)$. \\

For existing work on random walks on percolation clusters, including on-diagonal heat kernel estimates and invariance principles, see \cite{M+} and \cite{B-Perc}.  From now on, we fix $p>p_c(d)$, and write $q^\omega_t(x,y)$ for the heat kernel of the CSRW on $\mathcal{C}_{p,\infty}(\omega)$; the dependence on $\omega$ of $q^\omega_t(x,y)$ is a consequence of $\mathcal{C}_{p,\infty}(\omega)$ being random.  We denote the graph metric on $\mathcal{C}_{p,\infty}(\omega)$ by $d_\mathcal{C}$.  In \cite{M+}, Mathieu and Remy proved the following on-diagonal heat kernel bound for the CSRW on $\mathcal{C}_{p,\infty}(\omega)$.  

\begin{lem}
\cite{M+} There exist random variables $N_x(\omega)<\infty$ and non-random constants $c_1,c_2$ such that almost surely, for all $x\in G$ and $t>0$,

\begin{equation*}
\displaystyle q^\omega_t(x,x) \leq \begin{cases} c_1t^{-1/2} &\textnormal{if $0<t\leq N_x(\omega)$}, \\ c_2t^{-d/2} &\textnormal{if $N_x(\omega)<t$}.\end{cases}
\end{equation*}
\\
\end{lem}

The polynomial function $f(t) := c_2t^{d/2}$ is $(A,\gamma)-$regular on $(N_x(\omega),\infty)$ for $A=1$, $\gamma=2$, and hence an application of Theorem~\ref{HKUB3} shows that for $t\geq C(N_x(\omega)\vee N_y(\omega)) \vee 1\vee d_\mathcal{C}(x,y)$, we have the Gaussian upper bound

\begin{equation}
q^\omega_t(x,y) \leq C_1t^{-d/2}\exp\left(-C_2\frac{d_\mathcal{C}^2(x,y)}{t}\right), \label{HKPerc}
\end{equation}
\\
where $C_1,C_2>0$ are non-random constants. \\

{\bf Remarks:} \\

{\bf 1.} For the discrete time simple random walk on $\mathcal{C}_{p,\infty}(\omega)$, Gaussian upper bounds are obtained in \cite{C+} as an application of their discrete time heat kernel estimates.  However, the bounds in \cite{C+} have a random constant $C_1 = C_1(\omega)$ in \eqref{HKPerc}.  The reason is that \cite{C+} only considers functions which are $(A,\gamma)-$regular, and in general the function $f(t):=c_1^{-1}t^{1/2}{\bf 1}_{\{0<t\leq N_x(\omega)\}}+c_2^{-1}t^{d/2}{\bf 1}_{\{N_x(\omega)<t\}}$ is not $(A,\gamma)-$regular.  The authors of \cite{C+} therefore bound $f(t)$ by a smaller random function $g(t):=d_1t^{1/2}{\bf 1}_{\{0<t\leq N_x(\omega)\}}+d_2t^{d/2}{\bf 1}_{\{N_x(\omega)<t\}}$, where $d_1=d_1(\omega)$ and $d_2=d_2(\omega)$ are random constants chosen to ensure that $f\geq g$ and $g$ is $(A,\gamma)-$regular. \\

{\bf 2.} Theorem~\ref{HKUB3} is also used in \cite{A+} to obtain Gaussian upper bounds for the heat kernel in the random conductance model; as in the case of supercritical percolation clusters, the function appearing in the on-diagonal heat kernel estimate of Proposition 4.1 of \cite{A+} is not $(A,\gamma)-$regular but rather $(A,\gamma)-$regular on $(T,\infty)$ for some $T>0$, so Theorem~\ref{HKUB3} yields Gaussian upper bounds for all sufficiently large times. \\

\end{document}